\documentclass[12pt,leqno,a4paper]{article}

\usepackage[english]{babel}
\usepackage{amsmath}
\usepackage{amsfonts}
\usepackage{amsthm}
\usepackage{amssymb}

\usepackage{xspace}
\usepackage{euscript}
\usepackage{graphicx}
\usepackage{amscd}
\usepackage{tabularx}

\usepackage{enumerate}
 \usepackage{epsfig} 
 \usepackage{graphics} 
\theoremstyle{plain}
\newtheorem{theo}{Theorem}[section]
\newtheorem{prop}[theo]{Proposition}
\theoremstyle{definition}
\theoremstyle{remark}
\newtheorem{rem}[theo]{Remark}
\numberwithin{equation}{section}

\newcommand{\R}{\mathbb{R}}

\title{A straightforward proof of Carleman estimate for second order elliptic operator and a three sphere inequality}

\author{Lorenzo Baldassari\thanks{Department of Mathematics, ETH Zurich, Switzerland, E-mail:
\textsf{lorenzo.baldassari@sam.math.ethz.ch}},\quad\mbox{and}\quad
Sergio Vessella\thanks{Dipartimento di Matematica e Informatica, Universit\`a degli Studi di Firenze, Italy, E-mail:
\textsf{sergio.vessella@unifi.it}}}
\date{}

\begin{document}

\setcounter{section}{0}
\setcounter{secnumdepth}{2}

\maketitle

\begin{abstract}
In this paper we provide a simple proof of a Carleman estimate for a second order elliptic operator $P$ with Lipschitz leading coefficients. We apply such a Carleman estimate to derive a three sphere inequality for solutions to equation $Pu=0$.
\medskip

\noindent \textbf{Keywords:} Carleman estimate, Elliptic Equations, Propagation of smallness

\medskip

\noindent\textbf{Mathematics Subject Classification (2010)}: 35R25, 35J15, 35B45.

\end{abstract}

\section{Introduction} \label{sec:intro}

The main purpose of the present paper consists in providing a straightforward proof of Carleman estimates for second order elliptic operators with real coefficients in the principal part. Carleman estimates are among the most useful tools for studying  the unique continuation property or for proving uniqueness and stability for Cauchy problems for partial differential equations.

We recall that a linear partial differential equation $Pu=0$, enjoys a unique continuation property (UCP) in $\Omega\subset\mathbb{R}^{n}$, where $\Omega$ is a connected open set of $\mathbb{R}^{n}$, if the following property holds true \cite{hTat}:
for any open subset $A$ of $\Omega$
\begin{equation}\label{intr-1}
Pu=0\mbox{ in }\Omega \mbox{ and }u=0\mbox{ in }A\quad
\mbox{ imply }\quad u=0\mbox{ in }\Omega.
\end{equation}

 Furthermore, we call quantitative estimate of unique continuation (QEUC) or stability estimate related to the UCP property \eqref{intr-1} the
following type of result:
\begin{equation}\label{intr-1-1}
Pu=0\mbox{ in }\Omega\mbox{, } u \mbox{ small in } A \mbox{ and } \left\Vert u\right\Vert\leq 1 \mbox{ imply } u \mbox{ small in }\widetilde{\Omega},
\end{equation}
where $\left\Vert\cdot\right\Vert$ is a norm of a suitable space of functions defined on $\Omega$ and $\widetilde{\Omega}\Subset\Omega$. A classical example of an equation that enjoys the UCP is the Laplace equation $\Delta u=0$. In this case, property \eqref{intr-1} is an immediate consequence of the analyticity of solution $u$. In addition, if $\left\Vert\cdot\right\Vert=\left\Vert\cdot\right\Vert_{L^2(\Omega)}$, property  \eqref{intr-1-1} is also true, \cite{A-R-R-V}. It is evident that the main interest of the research in the topic of UCP concerns those equations whose solutions may be nonanalytic. For instance, this is the case of the equation $\Delta u+q(x)u=0$, where $q$ is a bounded function. For such equation, in two dimension Carleman \cite{ca} derived the estimates that are known as Carleman Estimates. Since Carleman's paper a vast literature has been developed on that and on its possible applications. Here we limit ourselves to refer to \cite{hormanderbook1} (see also \cite[Chapter 28]{ho3}) for general theory and to \cite{KlTi} and \cite[Chapter 3]{isakovlib2} for applications to inverse problems.

In this paper we provide, for a second order elliptic operator $P$, a simple proof of the Carleman estimate \cite[Theorem 8.3.1]{hormanderbook1}, that reads as follows:
there exist $C>0$ such that for every $u \in C^{\infty}_0(\Omega)$ and every $\tau\geq C$ we have
\begin{equation}\label{carleman-intr}
\tau \int_{\Omega} (|\nabla u|^2 + \tau^2 u^2)e^{2 \tau \phi} \leq C\int_{\Omega} |Pu|^2 e^{2 \tau \phi},
\end{equation}
where $\phi\in C^2(\overline{\Omega})$, is suitable weight function such that $\nabla \phi(x)\neq 0$, for every $x\in\overline{\Omega}$, and
\begin{equation*}
Pu=\partial_i(g^{ij} \partial_i u),
\end{equation*}
(we use the convention on repeated indices) where $\{g^{ij}(x)\}_{i,j=1}^n$ is a real-valued symmetric  $n\times n$ matrix that satisfies a hypothesis of uniform ellipticity whose entries are Lipschitz continuous (see Theorem \ref{main theo} for a precise statement).

Our proof is based on integration by parts and, in particular, on Rellich identity (see \eqref{rellich} below). Let us point out the that the application of Rellich identity in the context of Carleman estimates has been already employed in \cite{Ga}, \cite{GaSh}, see also (for the parabolic context) \cite{EsVe}. In order to illustrate our approach consider, for the sake of brevity, the simplest case, where $P=\Delta$, and let $v=e^{\tau \phi} u$. Very roughly speaking the main efforts of the proof consists in estimating from below the right hand side of \eqref{carleman-intr} by a definite positive quadratic form of the following variables:

i) $\tau|\nabla \phi| v$,

ii) the normal derivative to level surface $\{\phi=const\}$, that is the derivative of $v$ with respect to the direction of $\nabla \phi$,

iii) the tangential component of $\nabla v$ with respect to $\{\phi=const\}$.

\noindent To do this, we carry out a careful analysis of the strong pseudoconvexity conditions satisfied by weight function $\phi$ (see \eqref{ASSUMPTION} below) introduced in \cite{hormanderbook1}, see also \cite[Chapter 3]{isakovlib2}.

We refer the interested reader to \cite{EsVe} and \cite[Section 4]{MRV1} as papers strictly related to this note. Also we mention the nice notes \cite{Sa} in which, however, the Carleman estimate (for $P=\Delta$) is proved for the special type of weight functions $\phi=e^{\mu \psi}$, where $\nabla \psi\neq 0$ in $\overline{\Omega}$ and $\mu$ is large enough. Such a result can be obtained by combining Theorem \ref{main theo} and Proposition \ref{proposition} of the present paper.

We prove a three sphere inequality (see Theorem \ref{trsf} for a precise statement) for solutions to second order elliptic equation
\begin{equation}\label{eq-ell}
\partial_i(g^{ij} \partial_i u)=b_j\partial_j u+au,
\end{equation}
where $a\in L^{\infty}$ and $b_j\in L^{\infty}$, $j=1\cdots n$.
The proof of such three sphere inequality is quite standard and we give it mostly for the reader convenience and for completeness. As a matter of fact the three sphere inequality furnishes the "building brick" to the QUEC for solution to \eqref{eq-ell}, see \cite{A-R-R-V}.

The plan of the paper is as follows. In Section \ref{sec:notation} we collect some notations. In Section \ref{sec:theorem} we give the statement of the main theorem and we examine the main assumption. In Section \ref{sec:proof}  we provide the proof of the main theorem. In Section \ref{sec:tre sfere} we prove a three sphere inequality.

\section{Notation} \label{sec:notation}
Let $n\in\mathbb{N}$, $n\geq 2$. For any $\xi, \eta \in \R^n$, we define the following standard inner product $\langle \xi, \eta \rangle_n = \xi_i \eta_i$. Consequently we define $|\xi|_n=\sqrt{\langle \xi, \xi \rangle_n}$. We denote, as usual, by $|a|$ the absolute value of any real number $a$. Let $g(x)=\{g_{ij}(x)\}_{i,j=1}^n$, be a real-valued symmetric  $n\times n$ matrix  that for given constants $\lambda\geq 1$  and $\Lambda>0$ satisfies the following conditions (we use the convention on repeated indices)
\begin{equation}\label{elliptic}
\lambda^{-1} |\xi|^2_n \leq g_{ij}(x) \xi_i \xi_j \leq \lambda |\xi|^2_n \mbox{, for every } x,\xi\in \mathbb{R}^n
\end{equation}
and
\begin{equation}\label{lipschitz}
\sum_{j,i=1}^n | g_{ij}(x) - g_{ij}(y) | \leq \Lambda |x-y|_n \mbox{, for every } x,y\in \mathbb{R}^n.
\end{equation}
For any $\xi \in \R^n$, we define the vector $\xi^{(g)}=(g^{ij}(x)\xi_j)_{i=1}^n$, where $\{g^{ij}(x)\}_{i,j=1}^n$ is the inverse matrix of $g(x)$. We will use the following notations when considering either the vector fields $\xi$ and $\eta$ or the functions $v$ and $f$:
\begin{equation*}
\langle \xi, \eta \rangle = g_{ij}(x) \xi_i \eta_j \mbox{ , }\quad\quad |\xi|^2=\langle \xi, \xi \rangle,
\end{equation*}
\begin{equation*}
\nabla_g f(x) = g^{-1}(x) \nabla f (x) \mbox{ , }\quad\quad \Delta_g v= \mbox{div}(\nabla_g v) = \partial_i(g^{ij} \partial_i v).
\end{equation*}
Also, we denote by $(\nabla_{g} v(x))_k$ the k-component of the vector $\nabla_g v(x)$. Notice that
\begin{equation*}
\langle \xi^{(g)}, \eta^{(g)} \rangle = g^{ij}(x) \xi_i \eta_j = \langle \xi^{(g)}, \eta \rangle_n = \langle \xi, \eta^{(g)} \rangle_n,
\end{equation*}
\begin{equation*}
\langle \nabla_g v(x), \nabla_g f(x) \rangle = g^{ij}(x) \partial_i v (x) \partial_j f (x).
\end{equation*}
 We denote by $P(x, \partial)=g^{ij}(x) \partial^2_{ij} v (x)$ the principal part of the operator $\Delta_g$, consequently we denote
 $$P(x,\xi)=g^{ij}(x)\xi_i\xi_j\mbox{,  }\quad \xi\in \mathbb{R}^n.$$

\section{The Main Theorem} \label{sec:theorem}

In order to state Theorem \ref{main theo} below let us introduce some additional notations. Let $\Omega\subset \mathbb{R}^n$ be a bounded open set and let $\phi\in C^{2}(\bar{\Omega})$ be a real function such that
\begin{equation*}
\nabla \phi(x) \neq 0\mbox{, for every } x \in \bar{\Omega}.
\end{equation*}
Denote by
\begin{equation}\label{bound-phi}
m=\min_{\bar{\Omega}}|\nabla \phi| \quad\mbox{ and }\quad  M=\left\Vert\phi\right\Vert_{C^{2}(\bar{\Omega})}.
\end{equation}

In what follows we use the following identity (it is an immediate consequence of $\delta_h^i=g^{ij}g_{jh}$)
\begin{equation}
\partial_l g^{ik}=-g^{ij} (\partial_l g_{jh}) g^{hk}.
\label{identit�}
\end{equation}

For every $\xi\in\R^n$ and $\tau \neq 0$ we denote
\begin{equation*}
Q (x, \xi,\tau):=\partial^2_{jk} \phi(x)P^{(j)}(x, \zeta) \overline{P^{(k)}(x, \zeta)} + \tau^{-1} \mbox{ Im} \left(P_{(k)}(x, \zeta) \overline{P^{(k)} (x, \zeta)}\right),
\end{equation*}
where $\zeta=\xi+i \tau \nabla \phi(x)$,
\begin{equation*}
P^{(l)}(x, \xi) := \frac{\partial P}{\partial \xi_l} (x, \xi) = 2 g^{lk}(x)\xi_k, \quad \mbox{ for } l=1,\cdots, n
\end{equation*}
and
\begin{equation*}
P_{(s)}(x, \xi) := \frac{\partial P}{\partial x_s} (x, \xi) = \partial_s g^{jk}(x) \xi_j \xi_k, \quad \mbox{ for } s=1,\cdots, n.
\end{equation*}
Observe that by \eqref{identit�} we have
\begin{equation*}
P_{(s)}(x, \xi)= -\partial_s g_{kh}(x) \xi^{(g),k} \xi^{(g),h}, \quad \mbox{ for } s=1,\cdots, n,
\end{equation*}
where $\xi^{(g),k}$, $k=1,\cdots, n$, is the $k$-component of $\xi^{(g)}$.
Hence
\begin{equation}
\begin{aligned}
Q(x, \xi,\tau) =    \; & 4 [\partial^2_{jk} \phi \xi^{(g),j} \xi^{(g),k} + \tau^2 \partial^2_{jk} \phi (\nabla_g \phi)_j (\nabla_g \phi)_k ] - 4 (\partial_s g_{kh}) (\nabla_{g} \phi)_k \xi^{(g),h} \xi^{(g),s} \\ & +2(\partial_s g_{tw}) \xi^{(g),t} \xi^{(g),w} (\nabla_g \phi)_s -  2 \tau^2 (\partial_s g_{kh}) (\nabla_g \phi)_k (\nabla_g \phi)_h (\nabla_g \phi)_s.
\end{aligned}
\label{nuovadefQ}
\end{equation}
Notice that $Q(x, \xi,\tau)$ can be continuously defined also for $\tau=0$, hence from now on we consider $Q(x, \xi,\tau)$ defined for all $\tau\in \mathbb{R}$.

\bigskip

\noindent \textbf{Main Assumption}. In what follows we assume that there exists a constant $c_0>0$ such that
\begin{equation}
\label{ASSUMPTION}
\begin{aligned}
\begin{cases}
P(x, \zeta) = 0 \\
\zeta=\xi+i \tau \nabla \phi(x) \neq 0
\end{cases} \Longrightarrow   Q (x, \xi,\tau) \geq c_0|\zeta|^2.
\end{aligned}
\end{equation}
\begin{theo} \label{main theo}
If (\ref{ASSUMPTION}) holds true then there exist $C>0$ and $\tau_0>0$ such that for every $u \in C^{\infty}_0(\Omega)$ and every $\tau\geq \tau_0$
\begin{equation}\label{carleman}
\tau \int_{\Omega} (|\nabla_g u|^2 + \tau^2 u^2)e^{2 \tau \phi} \leq C\int_{\Omega} |\Delta_g u|^2 e^{2 \tau \phi}.
\end{equation}
\end{theo}
\begin{rem}
If $g_{ij} \in C^1(\R^n)$ $\forall i,j\in\{1,...,n\}$, then assumption \eqref{ASSUMPTION} is equivalent to
\begin{equation*}
\begin{aligned}
\begin{cases}
P(x, \zeta) = 0 \\
\zeta=\xi+i \tau \nabla \phi(x) \neq 0
\end{cases} \Longrightarrow   Q (x, \xi,\tau) >0,
\end{aligned}
\end{equation*}
that is the assumption (8.3.2) of \cite[Theorem 8.3.1]{hormanderbook1}.
\end{rem}

\begin{rem}
Since
\begin{equation*}
P(x,\xi+i \tau \nabla \phi) = |\xi^{(g)}|^2 + 2 i \tau \langle \xi^{(g)},\nabla_g \phi \rangle - \tau^2 |\nabla_g \phi |^2
\end{equation*}
and the matrix $g$ is positive definite, we have that the conditions $\nabla \phi (x) \neq 0$ and $\xi + i \tau \nabla \phi (x) \neq 0$ are equivalent, respectively, to
\begin{equation*}
\nabla_g \phi (x)\neq 0, \quad \xi^{(g)} + i \tau \nabla_g \phi (x)\neq 0.
\end{equation*}
Consequently, (\ref{ASSUMPTION}) is equivalent to
\begin{equation}
\begin{aligned} \begin{cases}
\nabla_g \phi \neq 0 \\
|\xi^{(g)}|^2=\tau^2|\nabla_g \phi|^2 \\
\tau \langle \xi^{(g)}, \nabla_g \phi \rangle =0 \\
\xi^{(g)}+i \tau \nabla_g \phi(x) \neq 0
\end{cases}  \Longrightarrow  && Q (x, \xi,\tau) \geq c_0(|\xi^{(g)}|^2 + \tau^2 |\nabla_g \phi|^2 ).
\end{aligned}
\label{new-assump}
\end{equation}
\end{rem}

\bigskip

Before proving Theorem \ref{main theo}, we re-write assumption (\ref{ASSUMPTION}) in a more convenient form.
Let us denote
\begin{equation}
N_g=\frac{\nabla_g \phi}{|\nabla_g \phi|}.
\label{N_g}
\end{equation}
Moreover we define, for $\vartheta\in \mathbb{R}^n$,
\begin{equation}
\begin{aligned}
q(x, \vartheta)  :=    \; 4 \partial^2_{jk} \phi \vartheta_j \vartheta_k - 4( \partial_s g_{kh}) (\nabla_{g} \phi)_k \vartheta_h \vartheta_s +2(\partial_s g_{tw}) (\nabla_g \phi)_s \vartheta_t \vartheta_w .
\end{aligned}
\label{Q1}
\end{equation}
Clearly, $q$ is homogeneous of degree $2$ with respect to the variable $\vartheta$. In addition it is easy to check that
\begin{equation}
\begin{aligned}
Q(x,\xi,\tau) =  q(x,\xi^{(g)})+q(x,\tau \nabla_g \phi)=q(x,\xi^{(g)})+\tau^2 |\nabla_g \phi|^2 q(x,N_g).
\end{aligned}
\label{Q somma Q1}
\end{equation}
and that assumption \eqref{new-assump} can be written as follows
\begin{equation}
\begin{aligned} \begin{cases}
|\xi^{(g)}|^2= \lambda^2\\
\langle \xi^{(g)}, N_g \rangle =0
\end{cases}  \Longrightarrow  q(x, \xi^{(g)})+ \lambda^2 q(x, N_g) \geq c_0 (|\xi^{(g)}|^2 + \lambda^2).
\end{aligned}
\label{third-assump}
\end{equation}
\section{Proof of Main Theorem} \label{sec:proof}
Let $u \in C^{\infty}_0(\Omega)$ and
\[v=e^{\tau \phi} u.\]
Let us define
\begin{gather}\label{defPtau}
P_{\tau} (x, \partial) v = e^{\tau \phi} \Delta_g (e^{-\tau \phi}v) =\\ \nonumber
\Delta_g v - 2\tau \langle \nabla_g \phi, \nabla_g v \rangle + (\tau^2 |\nabla_g \phi|^2 - \tau \Delta_g \phi) v.
\end{gather}
The main effort of the proof consists in proving that there exists $K$ and $\tau_0$, depending on $\lambda,\Lambda,m, M$ and $c_0$ only, such that for every $v\in C^{\infty}_0(\Omega)$ and every $\tau\geq \tau_0$ we have
\begin{equation}
\tau \int_{\Omega} (|\nabla_g v|^2 + \tau^2 |\nabla_g \phi|^2 v^2) \leq K\int_{\Omega} |P_{\tau} v|^2 .
\label{main effort}
\end{equation}
In order to prove this inequality we proceed in the following way.

\begin{description}
\item[First step] We find an estimate from below for $\int_{\Omega} |P_{\tau} v|^2$ by writing $P_\tau v$ as the sum of its symmetric and antisymmetric parts, $S_\tau v$ and $A_\tau v$ respectively. Moreover, we use the following \textbf{Rellich identity}
\vspace{3mm}
\begin{equation}
\begin{aligned}
&2 \langle B, \nabla_g f \rangle \Delta_g f = \mbox{div} [2 \langle B, \nabla_g f \rangle \nabla_g f - B |\nabla_g f|^2 ]\\
&+ \mbox{div}(B)|\nabla_g f|^2- 2 \partial_i B^k g^{ij} \partial_j f \partial_k f + B^k \partial_k g^{ij} \partial_i f \partial_j f.
\end{aligned}
\label{rellich}
\end{equation}
where $B \in C^1(\Omega, \R^n)$ and $f \in C^2(\Omega)$.

\item[Second step] In order to prove \eqref{main effort}, we apply main assumption \eqref{third-assump} to the estimate found in Step 1.

\item[Third step] We conclude the proof of Theorem \ref{main theo} by using $u=e^{\tau \phi} v$ in \eqref{main effort}.

\end{description}

\subsection{Step 1} \label{step1}
It is simple to check that operator $P^{\ast}_\tau v$, the formal adjoint of $P_\tau v$, is given by
\begin{equation*}
P^{\ast}_\tau v=\Delta_g v + \left(\tau^2 |\nabla_g \phi|^2+\tau (\Delta_g \phi)\right) v+2\tau \langle \nabla_g \phi, \nabla_g v \rangle,
\end{equation*}
hence the symmetric and antisymmetric parts of $P_\tau v$ are, respectively,
\begin{equation*}
S_{\tau} v =\frac{P_\tau v+P^{\ast}_\tau v}{2}= \Delta_g v + \tau^2 |\nabla_g \phi|^2 v,
\end{equation*}
\begin{equation*}
A_{\tau} v =\frac{P_\tau v-P^{\ast}_\tau v}{2}= -2\tau \langle \nabla_g \phi, \nabla_g v \rangle - \tau (\Delta_g \phi) v.
\end{equation*}
We have
\begin{equation}
\begin{aligned}
&\int_{\Omega} |P_{\tau} v|^2 =  \int_{\Omega} |S_{\tau} v+A_{\tau} v|^2\\ &=\int_{\Omega} |S_{\tau} v|^2 + \int_{\Omega} |A_{\tau} v|^2 + 2 \int_{\Omega} S_{\tau} v  A_{\tau} v.
\end{aligned}
\label{scomposizionePtau}
\end{equation}
We want to estimate from below the right-hand side of (\ref{scomposizionePtau}). We begin by $2 \int_{\Omega} S_{\tau} v  A_{\tau} v$. Notice that
\begin{equation*}
\begin{aligned}
2 \int_{\Omega} S_{\tau} v  A_{\tau} v = & -2 \tau \int_{\Omega} [2 \langle \nabla_g \phi, \nabla_g v \rangle \Delta_g v + 2 \tau^2 |\nabla_g \phi|^2 \langle \nabla_g \phi, \nabla_g v \rangle v \\ & + \Delta_g \phi (\Delta_g v) v + \tau^2 |\nabla_g \phi|^2 (\Delta_g \phi) v^2].
\end{aligned}
\end{equation*}
Using Rellich identity (\ref{rellich}) and the divergence theorem, we get
\begin{equation*}
\begin{aligned}
2 \int_{\Omega} S_{\tau} v  A_{\tau} v = & -2 \tau \int_{\Omega} [\Delta_g \phi |\nabla_g v|^2 - 2 \partial_i (\nabla_g \phi)_k g^{ij} \partial_j v \partial_k v + (\nabla_g \phi)_k \partial_k g^{ij} \partial_i v \partial_j v  \\ & + \tau^2 |\nabla_g \phi|^2 \langle \nabla_g \phi, \nabla_g (v^2) \rangle + \Delta_g \phi (\Delta_g v) v + \tau^2 |\nabla_g \phi|^2 (\Delta_g \phi) v^2].
\end{aligned}
\end{equation*}
Integrating by parts the term $\tau^2 |\nabla_g \phi|^2 \langle \nabla_g \phi, \nabla_g (v^2) \rangle$ we obtain
\vspace{3mm}
\begin{equation}
\begin{aligned}
2 \int_{\Omega} S_{\tau} v  A_{\tau} v = \; & \tau \int_{\Omega} \{ 4 [\partial_{il}^2 \phi (\nabla_g v)_i (\nabla_g v)_l + \partial_i g^{kl} \partial_l \phi \partial_k v (\nabla_g v)_i + \tau^2 \partial^2_{jr} \phi (\nabla_g \phi)_r (\nabla_g \phi)_j v^2 ] \\ & - 2 g^{k h} \partial_{h} \phi \partial_k g^{ij} \partial_i v \partial_j v + 2 \tau^2 (g^{ij} \partial_i \phi \partial_j g^{rs} \partial_r \phi \partial_s \phi) v^2 \} \\ & - 2 \tau \int_{\Omega} \Delta_g \phi |\nabla_g v|^2 - 2 \tau \int_{\Omega} \Delta_g \phi (\Delta_g v) v.
\end{aligned}
\label{SA}
\end{equation}
By \eqref{defPtau} we have
\begin{equation*}
\Delta_g v = P_{\tau} v + 2 \tau \langle \nabla_g \phi, \nabla_g v \rangle - \tau^2 |\nabla_g \phi |^2 v + \tau (\Delta_g \phi) v,
\end{equation*}
hence by \eqref{nuovadefQ} and \eqref{SA}, taking into account \eqref{identit�}, we have
\vspace{3mm}
\begin{equation}
\begin{aligned}
2 \int_{\Omega} S_{\tau} v  A_{\tau} v = \; & 4\tau^2 \int_{\Omega} \langle \nabla_g \phi, \nabla_g v \rangle^2 + \tau \int_{\Omega} Q (x, \nabla v,\tau v) \\ & - 2 \tau \int_{\Omega} \Delta_g \phi |\nabla_g v|^2 - 2 \tau \int_{\Omega} \Delta_g \phi (P_{\tau} v) v - 4 \tau^2 \int_{\Omega} \Delta_g \phi \langle \nabla_g \phi, \nabla_g v \rangle v \\ & + 2 \tau^3 \int_{\Omega} \Delta_g \phi |\nabla_g \phi |^2 v^2 - 2 \tau^2 \int_{\Omega} (\Delta_g \phi)^2 v^2.
\end{aligned}
\label{doppio prodotto}
\end{equation}
Now we estimate from below $\int_{\Omega} |S_{\tau} v|^2$. Let $\gamma \in C^{0,1}(\Omega)$ be a function that we are going to choose later. Integration by parts yields
\begin{equation*}
\int_{\Omega} \gamma (\Delta_g v) v = - \int_{\Omega} \left[\langle \nabla_g \gamma, \nabla_g v \rangle v + \gamma |\nabla_g v|^2 \right].
\end{equation*}
Then, we have
\begin{equation}
\begin{aligned}
\int_{\Omega} |S_{\tau} v|^2 & = \int_{\Omega} |(\Delta_g v + \tau^2 |\nabla_g \phi|^2 v - \tau \gamma v) + \tau \gamma v|^2 \\ & \geq 2 \tau \int_{\Omega} \left[\gamma v \Delta_g v + \tau^2 \gamma v^2 |\nabla_g \phi|^2 - \tau \gamma^2 v^2\right] \\ & = 2 \tau \int_{\Omega} \left[-\gamma (|\nabla_g v|^2 - \tau^2 v^2 |\nabla_g \phi|^2) - \tau \gamma^2 v^2 - \langle \nabla_g \gamma, \nabla_g v \rangle v\right].
\end{aligned}
\label{stima parte simmetrica}
\end{equation}
Notice that
\begin{equation}
\int_{\Omega} |A_{\tau} v|^2 = 4\tau^2 \langle \nabla_g \phi, \nabla_g v \rangle^2 + \tau^2 (\Delta_g \phi)^2 v^2 +4\tau^2 \langle \nabla_g \phi, \nabla_g v \rangle (\Delta_g \phi) v.
\label{parte antisimmetrica}
\end{equation}
Now we substitute (\ref{doppio prodotto}), (\ref{stima parte simmetrica}) and (\ref{parte antisimmetrica}) into (\ref{scomposizionePtau}) and, taking into account \eqref{identit�}, we get
\begin{equation}
\int_{\Omega} |P_{\tau} v|^2 \geq \int_{\Omega} \left(\widehat{Q}_{\tau}- 2 \tau \Delta_g \phi (P_{\tau} v) v\right),
\label{lower bound integrale}
\end{equation}
where
\begin{equation}
\begin{aligned}
\widehat{Q}_{\tau} = & \; 4\tau^2 \langle \nabla_g \phi, \nabla_g v \rangle^2 + \tau Q (x, \nabla v,\tau v) \\ & - 2 \tau (\Delta_g \phi + \gamma) (|\nabla_g v|^2 - \tau^2 |\nabla_g \phi|^2 v^2 ) - \tau^2 [(\Delta_g \phi)^2 + 2 \gamma^2 ] v^2 \\ & - 2 \tau \langle \nabla_g \gamma , \nabla_g v \rangle v .
\end{aligned}
\label{lower bound}
\end{equation}

\subsection{Step 2} \label{step2}
The main effort of the present step consists in using assumption  \eqref{ASSUMPTION} (in the form \eqref{third-assump}) in order to estimate from below $\widehat{Q}_{\tau}$ by a definite quadratic form.

Let $T_g$ be the tangential component of $\nabla_g v$, namely the component of $\nabla_g v$ orthogonal to $N_g$ (defined in \eqref{N_g}) with respect to $\langle \cdot \; , \cdot \rangle$. $T_g$ is given by
\begin{equation}
\begin{aligned}
T_g =  \nabla_g v - \langle \nabla_g v, N_g \rangle N_g.
\end{aligned}
\label{tangente}
\end{equation}
Denote by
\begin{equation}\label{XYZ}
X=|\langle \nabla_g v, N_g \rangle |, \; \; Y=|T_g|, \; \; Z = \tau |\nabla_g \phi| v,
\end{equation}
notice that
\begin{equation}
\begin{aligned}
|\nabla_g v|^2  = \langle \nabla_g v, N_g \rangle^2+|T_g|^2 =X^2+Y^2.
\end{aligned}
\label{norma-alpha}
\end{equation}

Since $\langle N_g, T_g \rangle = 0$, by applying \eqref{third-assump} to $\xi^{(g)}=T_g$ we have
\begin{equation}
\begin{aligned}
q(x,T_g) \geq c_0Y^2 - q(x,N_g) Y^2.
\end{aligned}
\label{result-assump}
\end{equation}
Now, denoting by $\{q_{hl}(x)\}^n_{h,l=1}$ the matrix associated to the quadratic form $q(x,\cdot)$, we have
\begin{equation}
\begin{aligned}
& q(x,\nabla_g v)  = q(x,T_g +\langle \nabla_g v, N_g \rangle N_g) \\ & = q (x,T_g) + \langle \nabla_g v, N_g \rangle^2 q (x,N_g) +  2  q_{hl}(x)  N_{g,h} T_{g,l}  \langle \nabla_g v, N_g \rangle.
\end{aligned}
\label{Q1 tangente}
\end{equation}
Since
\begin{equation}
\begin{aligned}
 q_{hl}  N_{g,h} T_{g,l}  \langle \nabla_g v, N_g \rangle & \geq  - |q_{hl}| |N_{g,h}| |T_{g,l}| |\langle N_g, \nabla_g v \rangle| \geq  - C_1X Y, \end{aligned}
\label{limitazione g}
\end{equation}
where $C_1$ depends on $\lambda$ and $\Lambda$ only, by \eqref{result-assump}, \eqref{Q1 tangente} and \eqref{limitazione g} we have

\vspace{3mm}
\begin{equation}
\begin{aligned}
q(x,\nabla_g v) + q (x,N_g) Z^2 \geq \; & c_0 Y^2 - q (x,N_g) Y^2 + X^2 q (x,N_g) \\ & + q (x,N_g) Z^2- 2  C_1X Y.
\end{aligned}
\label{new-result-assump}
\end{equation}

Denote
\begin{equation}
\alpha :=\Delta_g \phi + \gamma,
\label{alpha}
\end{equation}
and
\begin{equation*}
\begin{aligned}
F_{\alpha}(X,Y,Z)= \; \; & X^2 (4 \tau |\nabla_g \phi|^2 +q(x,N_g) - 2 \alpha )+Y^2(c_0 - q(x,N_g) - 2 \alpha ) \\ & + Z^2 (q(x,N_g)) + 2 \alpha) -2C_1 XY.
\end{aligned}
\end{equation*}
By \eqref{Q somma Q1}, \eqref{lower bound}, \eqref{norma-alpha} and \eqref{new-result-assump} we have

\vspace{3mm}
\begin{equation}
\begin{aligned}
\widehat{Q}_{\tau}\geq \tau F_{\alpha}(X,Y,Z) +R,
\end{aligned}
\label{result3}
\end{equation}
where
\begin{equation}
\begin{aligned}
R= -|\nabla_g \phi|^{-2} ( (\Delta_g \phi)^2 + 2 \gamma^2 ) Z^2 - 2 \langle \nabla_g \gamma , \nabla_g v  \rangle |\nabla_g \phi|^{-1} Z.
\end{aligned}
\label{result4}
\end{equation}
Notice that the matrix of quadratic form $F_{\alpha}$ is given by
\begin{equation}\label{matrix}
M=\begin{bmatrix}
4 \tau |\nabla_g \phi|^2 + q(x,N_g) - 2 \alpha & -C_1 & 0 \\ -C_1 & c_0 -q(x,N_g) - 2 \alpha & 0 \\ 0 & 0 & q(x,N_g) + 2 \alpha
\end{bmatrix}.
\end{equation}

Therefore $F_{\alpha}$ is positive definite if and only if all leading principal minors of $M$ are positive, that is if and only if we have
\begin{equation*}
\begin{aligned}
\begin{cases} q(x,N_g) + 2 \alpha >0 \\
\det M_1 >0 \\ \det M >0,
 \end{cases}
\end{aligned}
\end{equation*}
where
$$M_1=\begin{bmatrix} c_0 -q(x,N_g) - 2 \alpha & 0 \\ 0 & q(x,N_g) + 2 \alpha
\end{bmatrix}.$$
Now the first two conditions are satisfied if and only if
\begin{equation}
- \frac{q(x,N_g)}{2} < \alpha(x) < \frac{c_0 -q(x,N_g)}{2}.
\label{condizione alpha}
\end{equation}
By the \textbf{Main Assumption} we have $c_0>0$, hence condition \eqref{condizione alpha} is nonempty. Let $\widetilde{\alpha}$ be a function belonging to $C^{0,1}(\overline{\Omega})$ that satisfies \eqref{condizione alpha}. Also we may assume that
\begin{equation}\label{alpha-1}
\frac{c_0}{4}\leq q(x,N_g)+2\widetilde{\alpha}(x)\leq \frac{c_0}{2},\quad  x\in\Omega \mbox{ a.e. }
\end{equation}
and
\begin{equation}\label{alpha-2}
||\widetilde{\alpha}||_{C^{0,1}(\bar{\Omega})}\leq C_2,
\end{equation}
where $C_2$ depends on $\lambda,\Lambda,M$ and $c_0$ only. Now we choose $\alpha=\widetilde{\alpha}$. By  \eqref{elliptic}, \eqref{lipschitz}, \eqref{bound-phi}, \eqref{Q1} and \eqref{matrix} we have there exists $\tau_1$ depending on $\lambda,\Lambda,m, M$ and $c_0$ only such that if $\tau\geq \tau_1$ then

\begin{equation}\label{detM}
\det M\geq\frac{c^2_0m^2}{2\lambda}.
\end{equation}
Therefore by \eqref{alpha-1} and \eqref{detM} we have there exists $c_1>0$ depending on $\lambda,\Lambda,m, M$ and $c_0$ only such that if $\tau\geq \tau_1$ then

\begin{equation}
F_{\tilde{\alpha}}(X,Y, Z) \geq c_1 (X^2 + Y^2 + Z^2),
\label{(3-21.1)}
\end{equation}
By \eqref{norma-alpha}, \eqref{XYZ}, \eqref{result3} and \eqref{(3-21.1)} we have,

\begin{equation}\label{result4}
\widehat{Q}_{\tau}\geq \tau c_1 \left(|\nabla_g v|^2 + \tau^2 |\nabla_g \phi|^2 v^2\right)+R,
\end{equation}
for every $\tau \geq \tau_1$.
By \eqref{result4}, recalling that $Z = \tau |\nabla_g \phi| v$, we have
\begin{equation}\label{resto}
|R|\leq C_3\left(|\nabla_g v|^2 + \tau^2 |\nabla_g \phi|^2 v^2\right),
\end{equation}
where $C_3$ depends on $\lambda,\Lambda,m$ and $M$ only. Hence, by \eqref{result4}, \eqref{resto} and \eqref{lower bound integrale} there exists $\tau_2 \geq \tau_1$, $\tau_2$ depends on $\lambda,\Lambda,m, M$ and $c_0$ only such that for every $\tau \geq \tau_2$ we have

\begin{equation}
\int_{\Omega} |P_{\tau} v|^2 \geq \int_{\Omega} \left[\frac{c_1}{2}\left(\tau|\nabla_g v|^2 + \tau^3 |\nabla_g \phi|^2 v^2\right)- 2 \tau \Delta_g \phi (P_{\tau} v) v\right].
\label{lower-1}
\end{equation}
Now we have
\begin{equation}
 \left\vert 2 \tau \Delta_g \phi (P_{\tau} v) v\right\vert\leq C_4 \left\vert P_{\tau} v\right\vert^2+\tau^2|\nabla_g \phi|^2 v^2 .
\label{lower-2}
\end{equation}
where $C_4$ depends on $\lambda,\Lambda$ and $M$ only. Finally, let $\tau_0:=\max\left\{2c^{-1}_1,\tau_2\right\}$,  we get
\begin{equation}
\tau \int_{\Omega} (|\nabla_g v|^2 + \tau^2 |\nabla_g \phi|^2 v^2) \leq K\int_{\Omega} |P_{\tau} v|^2,
\label{main effort-1}
\end{equation}
for every $v\in C^{\infty}_0(\Omega)$ and every $\tau\geq \tau_0$, where $K$ depends on $\lambda,\Lambda,m, M$ and $c_0$ only.
\subsection{Step 3} \label{step3}
In this step we derive \eqref{carleman} by \eqref{main effort-1}.
Let $\epsilon \in (0,1)$ be a number that we will choose later. Since $v= e^{\tau \phi} u$, by \eqref{main effort-1} we have trivially
\begin{equation*}
\int_{\Omega} (\tau |\nabla_g v|^2 + \tau^3 |\nabla_g \phi|^2 v^2) \geq \epsilon  \int_{\Omega} \tau|\nabla_g u + \tau (\nabla_g \phi) u|^2 e^{2\tau \phi} + \int_{\Omega} \tau^3 |\nabla_g \phi|^2 u^2 e^{2\tau \phi}.
\end{equation*}
for  every $\tau\geq \tau_0$.

Now, by Young inequality we have
\begin{equation*}
\begin{aligned}
&  \epsilon\int_{\Omega}  (\tau|\nabla_g u + \tau (\nabla_g \phi) u|^2) e^{2\tau \phi} + \int_{\Omega} \tau^3 |\nabla_g \phi|^2 u^2 e^{2\tau \phi} \\ & = \epsilon  \int_{\Omega} \tau(|\nabla_g u|^2 + 2 \tau \langle \nabla_g \phi, \nabla_g u \rangle u + \tau^2 |\nabla_g \phi|^2 u^2) e^{2\tau \phi} + \int_{\Omega} \tau^3 |\nabla_g \phi|^2 u^2 e^{2\tau \phi}  \\ & \geq  \frac{\epsilon \tau}{2}  \int_{\Omega} |\nabla_g u|^2 e^{2\tau \phi}  + \tau^3 (1-\epsilon) \int_{\Omega} |\nabla_g \phi|^2 u^2 e^{2\tau \phi}.
\end{aligned}
\end{equation*}
Hence, for $\epsilon=\frac{1}{2}$, we have
\begin{equation}\label{quasifinito}
\begin{aligned}
& K\int_{\Omega} |\Delta_g u|^2 e^{2 \tau \phi}= K\int_{\Omega} |P_{\tau} v|^2  \geq  \int_{\Omega} (\tau |\nabla_g v|^2 + \tau^3 |\nabla_g \phi|^2 v^2) \geq \\ & \geq \frac{\tau}{4}  \int_{\Omega} |\nabla_g u|^2 e^{2\tau \phi}  + \frac{\tau^3}{2}  \int_{\Omega} |\nabla_g \phi|^2 u^2 e^{2\tau \phi}.
\end{aligned}
\end{equation}
By \eqref{elliptic}, \eqref{bound-phi} and \eqref{quasifinito} we obtain \eqref{carleman-intr}.  $\square $

\bigskip

The following proposition is useful to construct a weight function $\phi$ that satisfies the Main Assumption with desired level sets.

\begin{prop} \label{proposition}
Assume that \eqref{elliptic} and \eqref{lipschitz} are satisfied. Let $m_0, M_0$ be positive numbers. Let $\psi\in C^2\left(\bar{\Omega}\right)$ satisfy

\begin{subequations}
\label{varphi}
\begin{equation}
\label{varphi-a}
\min_{x\in\bar{\Omega}}|\nabla_g\psi|\geq m_0,
\end{equation}
\begin{equation}
\label{varphi-b}
\left\Vert\psi\right\Vert_{C^2\left(\bar{\Omega}\right)}\leq M_0.
\end{equation}
\end{subequations}
Then there exists $C>0$, depending on $\lambda, \Lambda, M_0$ such that if $\mu\geq \frac{C}{m_0}$ then $\phi=e^{\mu\psi}$ satisfies \eqref{ASSUMPTION}.
\end{prop}
\begin{proof}
We prove that for $\mu$ large enough, $\phi=e^{\mu\psi}$ satisfies \eqref{new-assump} (equivalent to \eqref{ASSUMPTION}). We have
\begin{equation}\label{derivate}
\partial_j\phi=\mu\partial_{j}\psi e^{\mu\psi}\mbox{ , } \partial_{jk}\phi=\left(\mu\partial_{jk}\psi+\mu^2\partial_{j}\psi\partial_{k}\psi\right)e^{\mu\psi}.
\end{equation}
Since \eqref{varphi-a} holds true, it is enough to check that for $\mu$ large enough we have that there exists $c_0>0$ such that, whenever $\tau\neq 0$ and
\begin{equation}\label{xi-phi}
|\xi^{(g)}|^2=\tau^2\mu^2|\nabla_g \psi|^2e^{2\mu\psi} \quad\mbox{ and }\quad  \langle \xi^{(g)}, \nabla_g \psi \rangle =0
\end{equation}
we have
\begin{equation}\label{psedocon}
Q (x, \xi,\tau) \geq c_0(|\xi^{(g)}|^2 + \tau^2 \mu^2|\nabla_g \psi|^2e^{2\mu\psi}),
\end{equation}
where $Q (x, \xi,\tau)$ is given by \eqref{nuovadefQ} (with $\phi=e^{\mu\psi}$).
Observe that by first equality in \eqref{xi-phi} we have

\begin{equation}\label{Qgeq}
Q (x, \xi,\tau) \geq 4Q_0-C_5\tau^2 \mu^3|\nabla_g \psi|^3e^{3\mu\psi},
\end{equation}
where $C_5$ depends on $\lambda$ and $\Lambda$ only and
\begin{equation*}
\begin{aligned}
 Q_0 &=e^{\mu\psi}\left(\mu\partial^2_{jk} \psi \xi^{(g),j} \xi^{(g),k}+\mu^2\langle \xi^{(g)}, \nabla_g \psi \rangle^2\right) \\ & +\tau^2 e^{3\mu\psi}\left(\mu^4|\nabla_g \psi|^4+\mu^3\partial^2_{jk} \psi (\nabla_g \psi)_j (\nabla_g \psi)_k\right).
\end{aligned}
\end{equation*}
Now, by \eqref{xi-phi} we have

\begin{equation*}
\begin{aligned}
 Q_0 =& e^{\mu\psi}\mu\partial^2_{jk} \psi \xi^{(g),j} \xi^{(g),k} +\tau^2 e^{3\mu\psi}\left(\mu^4|\nabla_g \psi|^4+\mu^3\partial^2_{jk} \psi (\nabla_g \psi)_j (\nabla_g \psi)_k\right)\\ &\geq \tau^2\mu^3|\nabla_g \psi|^2 e^{3\mu\psi}\left(\mu|\nabla_g \psi|^2-C_6M_0\right)\geq \tau^2\mu^3|\nabla_g \psi|^2 e^{3\mu\psi}\left(\mu m_0^2-C_6M_0\right)
\end{aligned}
\end{equation*}
where $C_6$ depends on $\lambda$ and $\Lambda$ only. Hence, for every $\mu\geq 2C_6M_0m_0^{-2}$ we have
\begin{equation}\label{Q-0}
Q_0\geq\frac{1}{2}\tau^2\mu^4m_0^2|\nabla_g \psi|^2 e^{3\mu\psi}.
\end{equation}
Now by \eqref{varphi-b}, \eqref{Qgeq} and \eqref{Q-0} we have, for every $\mu\geq M_0m_0^{-2}\max\{C_5,2C_6\}$,
\begin{equation}\label{Qgeq1}
Q (x, \xi,\tau) \geq \tau^2\mu^4m_0^2|\nabla_g \psi|^2 e^{3\mu\psi}
\end{equation}
and taking into account first equality in \eqref{xi-phi} we obtain \eqref{psedocon} with $c_0=\frac{1}{2}m_0^2\mu^2e^{\mu\Phi_0}$ and $\Phi_0=\min_{\bar{\Omega}}\psi$.

\end{proof}

\section{A three sphere inequality}\label{sec:tre sfere}

In this section we apply Carleman estimate \eqref{carleman} to prove a three sphere inequality for a solution $u$ to the equation
\begin{equation}\label{equation}
\Delta_g u = \langle b , \nabla_g u \rangle + a u \; \; \; \mbox{ in } B_1,
\end{equation}
where $B_1$ is the ball of $\mathbb{R}^n$ of radius $1$ centered at $0$, $b\in L^{\infty} (B_1,\mathbb{R}^n)$ and $a\in L^{\infty} (B_1,\mathbb{R})$. In addition, let $M_1$ be a given positive number, we assume

\begin{equation}\label{bound-coeff}
\left\Vert b\right\Vert_{L^{\infty} (B_1,\mathbb{R}^n)}\leq M_1 \quad\mbox{ and }\quad \left\Vert c\right\Vert_{L^{\infty} (B_1,\mathbb{R})}\leq M_1.
\end{equation}

\begin{theo}[\textbf{Three Sphere Inequality}] \label{trsf}
Assume that \eqref{elliptic}, \eqref{lipschitz} and \eqref{bound-coeff} are satisfied. Let $r_0, \rho$ satisfy $r_0<\rho<1$. Let $u \in H^2_{loc}(B_1)$ be a solution to \eqref{equation} then
\begin{equation}\label{carleman}
\| u \|_{L^2(B_\rho)} \leq C   \| u\|^{\theta}_{L^2(B_{r_0})} \| u \|^{1-\theta}_{L^2(B_1)},
\end{equation}
where $C$ and $\theta$, $\theta\in (0,1)$, depend on $\lambda,\Lambda,m, M$ and $r_0$ only.
\end{theo}

\begin{proof}
Given $\phi_{\mu}(x)= e^{-\mu |x|^2}$, we have by Proposition \ref{proposition} ($\psi(x)=-|x|^2$) that there exist $\overline{C}_1, \; \tau_0$ and $\mu_0 >0$,
 $\overline{C}_1, \tau_0, \mu_0$ depending on $\lambda,\Lambda,m, M$ and $r_0$ only, such that the following Carleman estimate holds true
\begin{equation}
\begin{aligned}
\int_{B_1 \setminus \bar{B}_{\frac{r_0}{8}}} (\tau^3 w^2 + \tau |\nabla_g w|^2) e^{2 \tau \phi_{\mu_0}} \leq \overline{C}_1 \int_{B_1 \setminus \bar{B}_{\frac{r_0}{8}}} |\Delta_g w|^2 e^{2 \tau \phi_{\mu_0}} ,
\end{aligned}
\label{(5-1.1)}
\end{equation}
for every $w \in C^{\infty}_0 (B_1 \setminus \bar{B}_{\frac{r_0}{8}})$ and for every $\tau \geq \tau_0$.
Let $\widetilde{\eta}\in C_0^2(0,1)$ such that $0\leq\widetilde{\eta}\leq 1$, $\widetilde{\eta}\equiv 1$ in $\left(\frac{r_0}{2},\frac{1}{2} \right)$, $\widetilde{\eta}\equiv 0$ in $\left(0,\frac{r_0}{4}\right)\cup \left(\frac{2}{3},1\right)$, $|(d^k/dt^k) \widetilde{\eta}|\leq cr_0^{-k}$ in $\left(\frac{r_0}{4},\frac{r_0}{2}\right)$, $|(d^k/dt^k) \widetilde{\eta}|\leq c$ in $\left(\frac{1}{2},\frac{2}{3}\right)$, $k=0,1,2$, where $c$ is a constant. Let us denote $\eta(x)=\widetilde{\eta}(|x|)$. Since $C^{\infty}_0 (B_1 \setminus \bar{B}_{\frac{r_0}{8}})$ is dense in $H^2_0 (B_1 \setminus \bar{B}_{\frac{r_0}{8}})$, inequality \eqref{(5-1.1)} holds true for $w=\eta u$ and we have, for every $\tau\geq \tau_0$,

\begin{equation}
\begin{aligned}
\int_{B_1 \setminus \bar{B}_{\frac{r_0}{4}}} (\tau^3 (\eta u)^2 + \tau |\nabla_g (\eta u)|^2) e^{2 \tau \phi_{\mu_0}} \leq \overline{C}_1 \int_{B_1 \setminus \bar{B}_{\frac{r_0}{4}}} |\Delta_g (\eta u)|^2 e^{2 \tau \phi_{\mu_0}}.
\end{aligned}
\label{(5-2.2)}
\end{equation}
Recalling that $\eta=1$ in $B_{\frac{1}{2}} \setminus \bar{B}_{\frac{r_0}{2}}$, the left-hand side of (\ref{(5-2.2)}) can be estimate from below trivially and we have
\begin{equation}
\begin{aligned}
\int_{B_{\frac{1}{2}} \setminus \bar{B}_{\frac{r_0}{2}}} (\tau^3 u^2 + \tau |\nabla_g u|^2) e^{2 \tau \phi_{\mu_0}} \leq \overline{C}_1 \int_{B_1 \setminus \bar{B}_{\frac{r_0}{4}}} |\Delta_g (\eta u)|^2 e^{2 \tau \phi_{\mu_0}}.
\end{aligned}
\label{(5-4.1)}
\end{equation}
Now by \eqref{equation} we have
\begin{equation}
\begin{aligned}
\left\vert\Delta_g (u \eta)\right\vert & = \left\vert\left(\langle b , \nabla_g u \rangle + au\right) \eta + 2 \langle \nabla_g \eta , \nabla_g u \rangle + u \Delta_g \eta\right\vert \\ & \leq \overline{C}_2 \left(\left\vert\nabla_g u\right\vert+|u|\right)\chi_{B_{\frac{2}{3}}\setminus \bar{B}_{\frac{r_0}{4}}},
\end{aligned}
\label{(5-3.1)}
\end{equation}
where $\chi_{B_{\frac{2}{3}}\setminus \bar{B}_{\frac{r_0}{4}}}$ is the characteristic function of the set $B_{\frac{2}{3}}\setminus \bar{B}_{\frac{r_0}{4}}$ and $\overline{C}_2$ depends on $\lambda,\Lambda,m, M, M_1$ and $r_0$ only.

Then, by  \eqref{(5-4.1)} and \eqref{(5-3.1)} we have
\begin{equation}
\begin{aligned}
& \int_{B_{\frac{1}{2}} \setminus \bar{B}_{\frac{r_0}{2}}} (\tau^3 u^2 + \tau |\nabla u|^2) e^{2 \tau \phi_{\mu_0}} \leq \overline{C}_3\int_{B_{\frac{2}{3}} \setminus \bar{B}_{\frac{r_0}{4}}}  \left(u^2 + |\nabla_g u|^2\right) e^{2 \tau \phi_{\mu_0}}\\ & =\overline{C}_3 \int_{B_{\frac{1}{2}} \setminus \bar{B}_{\frac{r_0}{2}}}  \left(u^2 + |\nabla_g u|^2\right) e^{2 \tau \phi_{\mu_0}} \\ &+ \overline{C}_3\int_{B_{\frac{r_0}{2}} \setminus \bar{B}_{\frac{r_0}{4}}}  \left(u^2 + |\nabla_g u|^2\right) e^{2 \tau \phi_{\mu_0}}  + \overline{C}_3\int_{B_{\frac{2}{3}} \setminus \bar{B}_{\frac{1}{2}}} \left(u^2 + |\nabla_g u|^2\right) e^{2 \tau \phi_{\mu_0}},
\end{aligned}
\label{(5-6.1)}
\end{equation}
where $\overline{C}_3$ depends on $\lambda,\Lambda,m, M, M_1$ and $r_0$ only.
By (\ref{(5-6.1)}) we have

\begin{equation}
\begin{aligned}
& \left(\tau-\overline{C}_3\right)\int_{B_{\frac{1}{2}} \setminus \bar{B}_{\frac{r_0}{2}}} (|\nabla_g u|^2 + u^2 )e^{2 \tau \phi_{\mu_0}} \leq \\ &  \overline{C}_3\int_{B_{\frac{r_0}{2}} \setminus \bar{B}_{\frac{r_0}{4}}} \left(u^2 + |\nabla_g u|^2\right) e^{2 \tau \phi_{\mu_0}}+\overline{C}_3\int_{B_{\frac{2}{3}} \setminus \bar{B}_{\frac{1}{2}}} \left(u^2 + |\nabla_g u|^2\right) e^{2 \tau \phi_{\mu_0}}
\end{aligned}
\label{(5-8.1)}
\end{equation}
Hence, for every $\tau\geq\overline{\tau}_1:=\max\{2\overline{C}_3,\tau_0\}$ we have
\begin{equation}
\begin{aligned}
& \int_{B_{\frac{1}{2}} \setminus \bar{B}_{\frac{r_0}{2}}} u^2 e^{2 \tau \phi_{\mu_0}} \leq   \int_{B_{\frac{r_0}{2}} \setminus \bar{B}_{\frac{r_0}{4}}} \left(u^2 + |\nabla_g u|^2\right) e^{2 \tau \phi_{\mu_0}}+\int_{B_{\frac{2}{3}} \setminus \bar{B}_{\frac{1}{2}}} \left(u^2 + |\nabla_g u|^2\right) e^{2 \tau \phi_{\mu_0}}.
\end{aligned}
\label{(5-8.1)-1}
\end{equation}
Since $$[0,+\infty)\ni t\rightarrow \widetilde{\phi}_{\mu_0}\left(t\right):=e^{-\mu_0 t^2}$$ is a decreasing function, \eqref{(5-8.1)-1} gives, for every $ \tau \geq \overline{\tau}_1$.
\begin{equation}
\begin{aligned}
 &\int_{B_{\frac{1}{2}} \setminus \bar{B}_{\frac{r_0}{2}}} u^2 e^{2 \tau \phi_{\mu_0}}  \leq e^{2 \tau \widetilde{\phi}_{\mu_0}\left(r_0/4\right)} \int_{B_{\frac{r_0}{2}} \setminus \bar{B}_{\frac{r_0}{4}}} \left(u^2 + |\nabla_g u|^2\right)\\ & +e^{2 \tau \widetilde{\phi}_{\mu_0}\left(1/2\right)}\int_{B_{\frac{2}{3}} \setminus \bar{B}_{\frac{1}{2}}} \left(u^2 + |\nabla_g u|^2\right),
\end{aligned}
\label{(5-10.1)}
\end{equation}

Now, by the Caccioppoli inequality \cite{caccioppoli}, see also \cite{Gi}, we have
\begin{equation}
\begin{aligned}
& \int_{B_{\frac{r_0}{2}} \setminus \bar{B}_{\frac{r_0}{4}}}  |\nabla_g u|^2 \leq \frac{\overline{C}_4}{r^2_0} \int_{B_{r_0} \setminus \bar{B}_{\frac{r_0}{8}}}  u^2,   \\
&\int_{B_{\frac{2}{3}} \setminus \bar{B}_{\frac{1}{2}}}  |\nabla_g u|^2   \leq \overline{C}_4 \int_{B_{\frac{5}{6}} \setminus \bar{B}_{\frac{1}{6}}}  u^2 ,
\end{aligned}
\label{(5-11.1)}
\end{equation}
where $\overline{C}_4$ depends on $\lambda $ only,
hence by \eqref{(5-10.1)} and \eqref{(5-11.1)} we get
\begin{equation}\label{17-9-17-1}
\begin{aligned}
& \int_{B_{\frac{1}{2}} \setminus \bar{B}_{\frac{r_0}{2}}} u^2 e^{2 \tau \phi_{\mu_0}} \\ & \leq \overline{C}_5 \left(e^{2 \tau \widetilde{\phi}_{\mu_0}\left(r_0/4\right)}  \int_{B_{r_0} \setminus \bar{B}_{\frac{r_0}{8}}}  u^2+ e^{2 \tau \widetilde{\phi}_{\mu_0}\left(1/2\right)}\int_{B_{\frac{5}{6}} \setminus \bar{B}_{\frac{1}{6}}}  u^2\right),
\end{aligned}
\end{equation}
where $\overline{C}_5$ depends on $\lambda,\Lambda,m, M, M_1$ and $r_0$ only.

Let $\rho$ be such that $\frac{r_0}{2} < \rho < \frac{1}{2}$. Trivially we have
\begin{equation}\label{17-9-17-2}
\int_{B_{\frac{1}{2}} \setminus \bar{B}_{\frac{r_0}{2}}} u^2 e^{2 \tau \phi_{\mu_0}}  \geq  \int_{B_{\rho} \setminus \bar{B}_{\frac{r_0}{2}}} u^2 e^{2 \tau \phi_{\mu_0}}\geq e^{2 \tau \widetilde{\phi}_{\mu_0} (\rho)} \int_{B_{\rho} \setminus \bar{B}_{\frac{r_0}{2}}} u^2.
\end{equation}
Since that $B_{r_0} \setminus \bar{B}_{\frac{r_0}{8}} \subset B_{r_0}$, $B_{\frac{5}{6}} \setminus \bar{B}_{\frac{1}{6}} \subset B_{1}$, by \eqref{17-9-17-1} and \eqref{17-9-17-2} we get:
\begin{equation}\label{17-9-17-3}
\begin{aligned}
\int_{B_{\rho} \setminus \bar{B}_{\frac{r_0}{2}}} u^2  \; dx \leq  \overline{C}_5 & \left( e^{2 \tau [\widetilde{\phi}_{\mu_0}(\frac{r_0}{4}) - \widetilde{\phi}_{\mu_0} (\rho)]} \int_{B_{r_0}}  u^2  \right. \\ & \left.  +e^{2 \tau [\widetilde{\phi}_{\mu_0}(\frac{1}{2})- \widetilde{\phi}_{\mu_0} (\rho)]} \int_{B_1} u^2  \right),
\end{aligned}
\end{equation}
for every $ \tau \geq \overline{\tau}_1$.
Now, add $\int_{B_{\frac{r_0}{2}}}  u^2 $ to both the sides of the inequality \eqref{17-9-17-3}, and by the trivial inequality  $\int_{B_{\frac{r_0}{2}}}  u^2 \leq \int_{B_{r_0}} u^2$, we obtain
\begin{equation}\label{(5-13.1)}
\begin{aligned}
\int_{B_{\rho}} u^2 \leq   & \overline{C}_6\left( e^{2 \tau [\widetilde{\phi}_{\mu_0}(\frac{r_0}{4}) - \widetilde{\phi}_{\mu_0} (\rho)]} \int_{B_{r_0}}  u^2  \right. \\ & \left.  +e^{2 \tau [\widetilde{\phi}_{\mu_0}(\frac{1}{2})- \widetilde{\phi}_{\mu_0} (\rho)]} \int_{B_1} u^2  \right).
\end{aligned}
\end{equation}
for every $ \tau \geq \overline{\tau}_1$, where $\overline{C}_6=\overline{C}_5+1$.

Now let us denote
\begin{equation}
\widetilde{\tau}= \frac{1}{2[\widetilde{\phi}_{\mu_0}(\frac{r_0}{4})-\widetilde{\phi}_{\mu_0}(\frac{1}{2})]} \log \left(\frac{\| u \|^2_{L^2(B_1)}}{\| u\|^2_{L^2(B_{r_0})}}\right)
\label{(4-14.1)}
\end{equation}
and notice that
\begin{equation}
e^{2 \widetilde{\tau} [\widetilde{\phi}_{\mu_0}(\frac{r_0}{4}) - \widetilde{\phi}_{\mu_0} (\rho)]} \| u\|^2_{L^2(B_{r_0})} = e^{2 \widetilde{\tau} [\widetilde{\phi}_{\mu_0}(\frac{1}{2})-\widetilde{\phi}_{\mu_0} (\rho)] }\| u \|^2_{L^2(B_1)}.
\label{(4-13.2)}
\end{equation}
If $\widetilde{\tau} \geq \overline{\tau}_1$, we choose $\tau= \widetilde{\tau}$ in \eqref{(5-13.1)} and we have
\begin{equation}
\begin{aligned}
\| u \|^{2}_{L^2(B_\rho)} \leq 2\overline{C}_6 \left(\| u \|_{L^2(B_{r_0})}\right)^{2\theta} \left(\| u\|_{L^2(B_1)}\right)^{2(1-\theta)} .
\end{aligned}
\label{(4-14.2)}
\end{equation}
If $\widetilde{\tau} \leq \overline{\tau}_1$, then by \eqref{(4-14.1)} we have trivially
\begin{equation}
\| u \|^2_{L^2(B_1)} < e^{2 \overline{\tau}_1 [\widetilde{\phi}_{\mu_0} (\frac{r_0}{4})-\widetilde{\phi}_{\mu_0} (\frac{1}{2})]} \| u\|^2_{L^2(B_{r_0})}.
\label{(4-15.1)}
\end{equation}
Hence we have
\vspace{3mm}
\begin{equation}
\begin{aligned}
\| u \|^2_{L^2(B_\rho)}\leq \| u \|^2_{L^2(B_1)}& =  (\| u \|_{L^2(B_1)})^{2\theta} (\| u \|_{L^2(B_1)})^{2(1-\theta)}\\& \leq e^{2 \overline{\tau}_1 [\widetilde{\phi}_{\mu_0}(\rho)-\widetilde{\phi}_{\mu_0}(\frac{1}{2})]} \left(\| u \|_{L^2(B_{r_0})}\right)^{2\theta} \left(\| u\|_{L^2(B_1)}\right)^{2(1-\theta)}.
\label{(4-16.2)}
\end{aligned}
\end{equation}
By \eqref{(4-14.2)} and \eqref{(4-16.2)} we have

\begin{equation}
\| u \|^2_{L^2(B_\rho)} \leq  \overline{C}_7\left(\| u \|_{L^2(B_{r_0})}\right)^{2\theta} \left(\| u\|_{L^2(B_1)}\right)^{2(1-\theta)}.
\label{sfere2}
\end{equation}
where $\overline{C}_7=\max\{2\overline{C}_6, e^{2\overline{\tau}_1 }\}$ and the proof is complete.
\end{proof}

\section*{Acknowledgement}
The author SV was partially supported by Gruppo Nazionale per l'Analisi Matematica, la Probabilit\`{a} e le loro Applicazioni (GNAMPA) of the Istituto Nazionale di Alta Matematica (INdAM).

\end{document}